\documentclass{alice.ams}

\newtheorem{thm}{Theorem}[section]

	\usepackage{amsmath,amssymb,graphicx,enumerate,bbm,latexsym,calc,capt-of,ifthen,mathabx}
	\usepackage[font=scriptsize]{caption}
	\usepackage{float}
	\newcommand{\Beta}{\mathrm{B}}
	\newcommand{\assign}{:=}
	
	\newcommand{\mathd}{\mathrm{d}}
	\newcommand{\tmop}[1]{\ensuremath{\operatorname{#1}}}
	
	\newenvironment{enumerateroman}{\begin{enumerate}[i.] }{\end{enumerate}}
	\newenvironment{proof*}[1]{\noindent\textit{#1\ }}{\hspace*{\fill}$\Box$\medskip}
	\newtheorem{lemma}[thm]{Lemma}
	\newtheorem{proposition}[thm]{Proposition}
	\newtheorem{theorem}[thm]{Theorem}
	\theoremstyle{definition}
	\newtheorem{remark}[thm]{Remark}
	\newtheorem{definition}[thm]{Definition}
	\newcommand{\tmfloatcontents}{}
	\newlength{\tmfloatwidth}
	\newcommand{\tmfloat}[5]{
		\renewcommand{\tmfloatcontents}{#4}
		\setlength{\tmfloatwidth}{\widthof{\tmfloatcontents}+1in}
		\ifthenelse{\equal{#2}{small}}
		{\setlength{\tmfloatwidth}{0.45\linewidth}}
		{\setlength{\tmfloatwidth}{\linewidth}}
		\begin{minipage}[#1]{\tmfloatwidth}
			\begin{center}
				\tmfloatcontents
				\captionof{#3}{#5}
			\end{center}
	\end{minipage}}
	
\begin{document}
	
		\title[Log-concavity of inv. inc. beta function wrt parameter]{Logarithmic concavity of the inverse incomplete
	beta function with respect to parameter}
	
	\author{Dimitris Askitis}
	\address{Department of Mathematical Sciences\\
		University of Copenhagen\\
		Universitetsparken 5\\
		Copenhagen 2100\\
		Denmark}
	\email{dimitrios@math.ku.dk}

		\maketitle
		
		\begin{abstract}
			The beta distribution is a two-parameter family of probability distributions
			whose distribution function is the (regularised) incomplete beta function.
			In this paper, the inverse incomplete beta function is studied analytically as univariate
			function of the first parameter. Monotonicity, limit results and convexity
			properties are provided. In particular, logarithmic concavity of the inverse
			incomplete beta function is established. In addition, we provide
			monotonicity results on inverses of a larger class of parametrised
			distributions that may be of independent interest.
		\end{abstract}
		
		\section{Introduction}
		
		Let a probability distribution on $I \subset \mathbbm{R}$ having cumulative
		distribution function (CDF) $F$. A median of it is defined as a point on $I$
		that leaves half of the ``mass'' on the left and half on the right, i.e. a
		value $m \in I$ such that $F (m) = 1 / 2$. In a similar way, we consider the
		more general notion of a $p$-quantile:
		\begin{definition}
			Let a probability distribution on $I\subset \mathbbm{R}$ with cumulative distribution function $F$, and $p
			\in (0, 1)$. A value $q_p \in I$ is a $p$-quantile of it if $F (q) = p$.
		\end{definition}			
		In this notation, the $1 / 2$-quantile is exactly the median. It is not always
		the case that a $p$-quantile exists for a probability distribution, or that it
		is unique. However, existence and uniqueness are guaranteed if $D$ has an a.e.
		positive density wrt Lebesgue measure. In this case, we may consider the
		inverse distribution function of $F$. The median and $p$-quantiles have
		importance in statistics as measures of position less affected by extreme
		values than e.g. the mean, and they have further uses considering levels of
		significance.
		
		We are interested in parametrised families of probability distributions and
		the behaviour of the $p$-quantile with respect to the parameter, with $p$
		being fixed. In case we have a family of cumulative distribution functions
		$F_a$, $a$ being the parameter of the family, such that for each $a$ the
		corresponding $p$-quantile exists and is unique, we may define it as a
		function of $a$ implicitly through the functional equation $F_a (q_p (a)) =
		p$.
		
		In the case of the median of the gamma distribution, such studies have been
		done in several occasions, e.g. in {\cite{gammaram1}}, {\cite{chenrubin0}} and
		{\cite{gammaram2}}. In {\cite{gammaram3}}, Adell and Jodr{\'a} explore a very
		interesting connection with a sequence by Ramanujan. In {\cite{chenrubin}} and
		{\cite{mediangammaconvex}}, Berg and Pedersen give a proof of the continuous
		version of the Chen-Rubin conjecture, originally stated in
		{\cite{chenrubin0}}, and they moreover prove convexity and find asymptotic
		expansions.
		
		In the present article, the main focus is on the $p$-quantile of the beta
		distribution, or equivalently the inverse of the (regularised) incomplete beta
		function (\ref{regincbeta}), as a function of the parameter $a$. This inverse
		has also been considered by Temme {\cite{temme}} who studied its uniform
		asymptotic behaviour. In particular, his results give a very accurate
		approximation for the inverse for $a + b > 5$. This is used in
		computer algorithms approximating the inverse incomplete beta function. Also,
		see {\cite{beta1}} for some interesting inequalities on the median. In \cite{KarpInc}, logarithmic convexity/concavity results are proved for the regularised incomplete beta function wrt to parameters, though the methods employed there are quite different, and there does not seem to be some direct connection with the results in the present article. In applications, (strict) logarithmic concavity is an important property, as it ensures the uniqueness of minimum and it is invariant under taking products.
		
		The beta function is defined as the ratio of gamma functions
		\begin{equation}
		\Beta (a, b) \assign \frac{\Gamma (a) \Gamma (b)}{\Gamma (a + b)}
		\end{equation}
		One also has an integral representation of the beta function for $a, b > 0$
		given by
		\begin{equation}
		\Beta (a, b) = \int_0^1 t^{a - 1} (1 - t)^{b - 1} \mathd t
		\end{equation}
		More information on the beta function can be found on
		{\cite{special_functions}}. The beta distribution is the $2$-parameter family
		of probability distributions, whose cumulative distribution function is the regularised
		incomplete beta function
		\begin{equation}
		I (x ; a, b) \assign \frac{\int_0^x t^{a - 1} (1 - t)^{b - 1} \mathd
			t}{\Beta (a, b)} \label{regincbeta}
		\end{equation}
		We fix $p \in (0, 1)$ and $b > 0$, and we consider the first parameter $a$ as
		a variable. We shall see in the Appendix that, due to a reflection formula for
		the regularised incomplete beta function, we can translate the results to the
		case that we fix the other parameter instead. We consider the $p$-quantile of the beta
		distribution, which in the literature is often also called the inverse incomplete beta function, as a function of $a$. We denote it by $q : (0,
		\infty) \rightarrow (0, 1)$ and define it implicitly by the
		equation $I (q (a) ; a, b) = p$, or equivalently by
		\begin{equation}
		\int_0^{q (a)} t^{a - 1} (1 - t)^{b - 1} \mathd t = p \int_0^1 t^{a - 1} (1
		- t)^{b - 1} \mathd t \label{betamaineq}
		\end{equation}
		In the literature this value is often denoted by $I^{-1}_p(a,b)$, and in our case $q$ is the function $a\mapsto I^{-1}_p(a,b)$. Moreover, we consider the function
		\begin{equation}
		\phi (a) \assign - a \log q (a) \label{phi}
		\end{equation}
		which turns out giving further information on $q$. In the
		following plots we can get an idea on how the median of the beta distribution
		behaves wrt $a$.
	\begin{figure}[H]\centering
		\includegraphics[height=6cm,keepaspectratio]{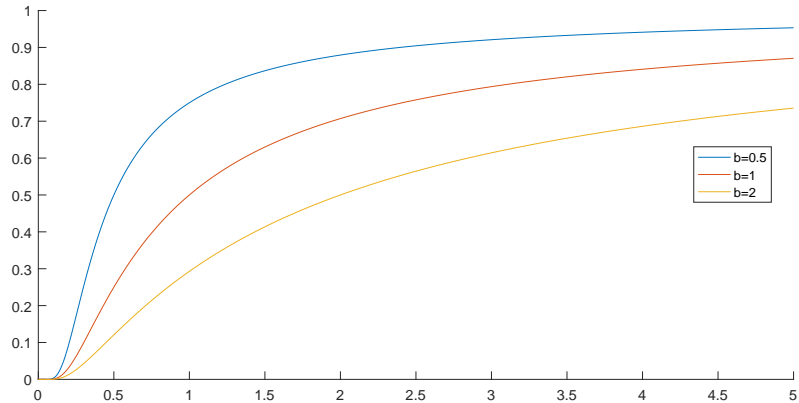}
		\caption{Plot of $q$\label{figq} for $p=1/2$}
	\end{figure}
\begin{figure}[H]\centering
			\includegraphics[height=6cm,keepaspectratio]{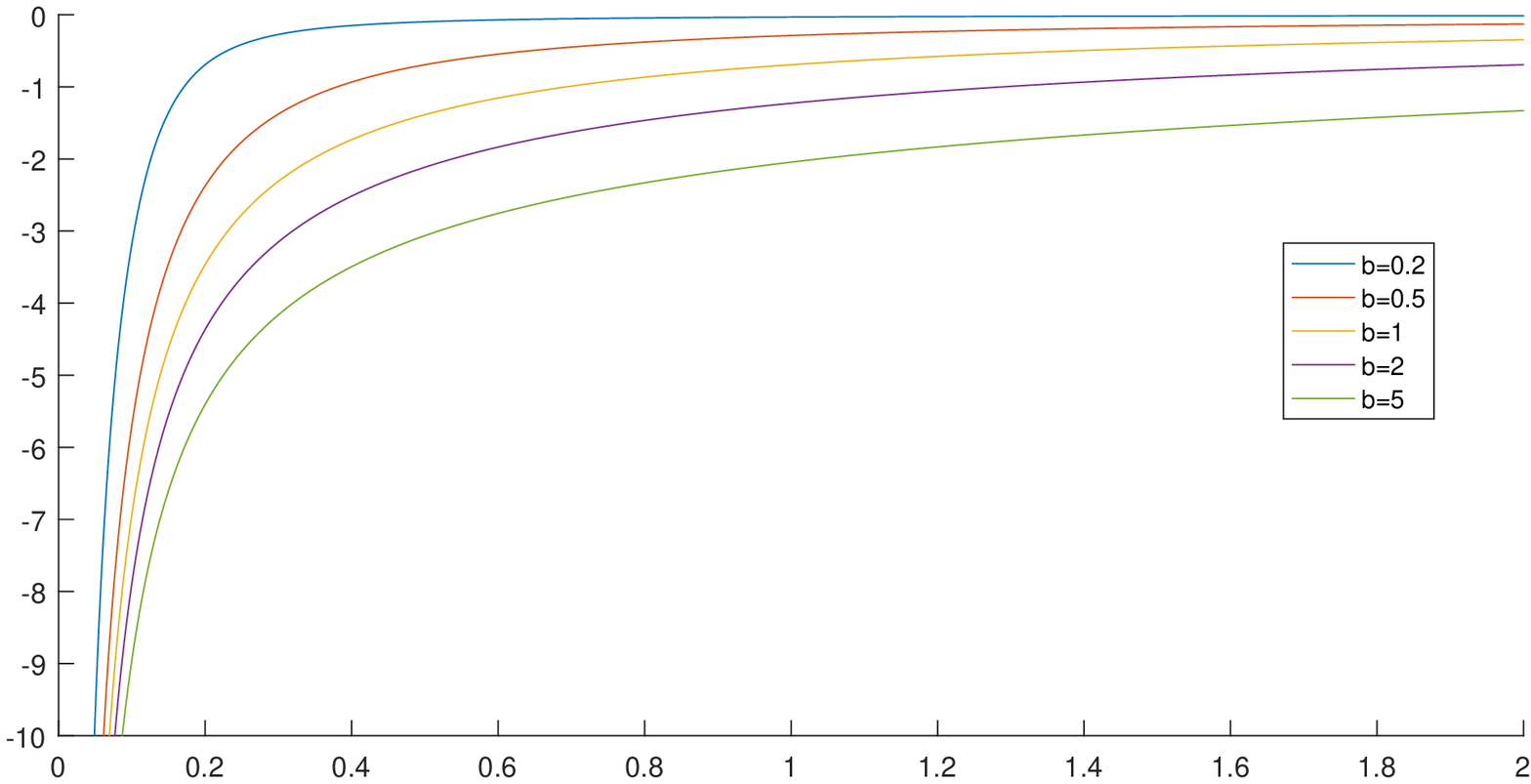}
	\caption{Plot
		of $\log q$ for $p=1/2$}\label{figlog}
\end{figure}
\begin{figure}[H]\centering
			\includegraphics[height=6cm,keepaspectratio]{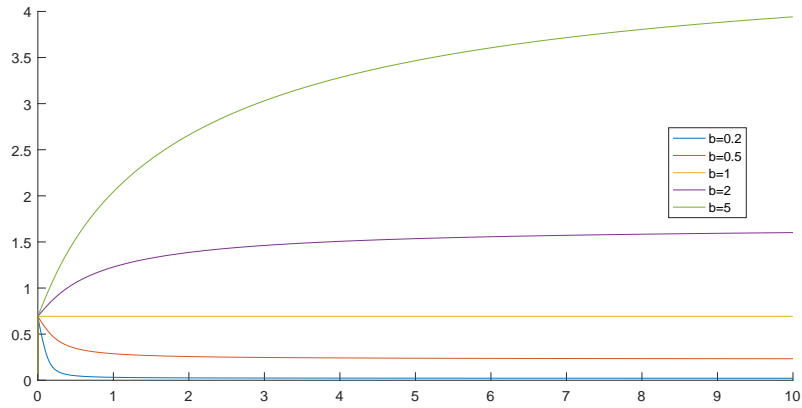}
	\caption{Plot
		of $\phi$ for $p=1/2$}\label{figphi}
\end{figure}
		
		In the rest of the paper we fix $p\in (0,1)$. We first get the following two propositions, regarding monotonicity and first order asymptotics:
		
		\begin{proposition}
			\label{prop1}The function $q$ in (\ref{betamaineq}) is a real
			analytic and increasing function on $(0, \infty)$. It has
			limits
			\[ \lim_{a \rightarrow 0} q (a) = 0 \]
			and
			\[ \lim_{a \rightarrow \infty} q (a) = 1 \]
		\end{proposition}
		
		\begin{proposition}
			\label{prop2}The function $\phi$ in (\ref{phi}) is real analytic on $(0,
			\infty)$. It is decreasing if $b < 1$, constant if $b = 1$ and increasing if
			$b > 1$. It has limits
			\[ \lim_{a \rightarrow 0} \phi (a) = - \log p \]
			and
			\[ \lim_{a \rightarrow \infty} \phi (a) = \gamma_b \]
			where $\gamma_b$ is the $(1 - p)$-quantile of the gamma distribution with
			parameter $b$.
		\end{proposition}
		
		Then, we investigate the analytic properties of the inverse
		incomplete beta function deeper. In particular, investigating its logarithm, we
		obtain the following two results, which consist the main contribution of this
		paper:
		
		\begin{theorem}
			\label{conv1}For fixed $b \in (0,1)$, $\phi$ in (\ref{phi}) is (strictly) convex.
		\end{theorem}
		
		\begin{theorem}
			\label{logconc}For fixed $b \in (0,\infty)$, $q$ in (\ref{betamaineq}) is (strictly) log-concave.
		\end{theorem}
		
		\begin{remark}
			One can infer from Figure \ref{figq} that $q$ is neither concave nor
			convex; its reciprocal $1 / q$, though, is logarithmically convex by Theorem
			\ref{logconc}, hence also convex. Moreover, based on Figure \ref{figphi}, as
			well as numerical results, for $b > 1$ we conjecture that $\phi$ is concave.
		\end{remark}
		
		The article is organised in the following way. In section 2 we present some
		general results regarding $p$-quantiles of more general probability
		distributions, that have some interest by themselves. For instance, Lemma
		\ref{lemmageratio} is a generalisation of results concerning monotonicity
		properties of ratios of power series and polynomials to ratios of integrals.
		In section 3 we study the monotonicity and limit properties of $q$ and $\phi$
		and prove Propositions \ref{prop1} and \ref{prop2}. In section 4 we prove
		convexity of $\phi$ for $b < 1$, while in section 5 we prove logarithmic
		concavity of $q$. In the Appendix, we look into the dependence on the
		parameter $b$ with $a$ being fixed and translate some of the results in this
		case.
		
		\section{General results on $p$-quantiles of probability distributions}
				
		The following lemma is a standard result in measure theory, that lets us
		interchange integration and differentiation {\cite[Theorem 6.28]{prob}}. In
		the rest of the paper, $\partial_x$ denotes differentiation with respect to the
		variable $x$.
		
		\begin{lemma}
			\label{lemmaderiv}Let $(\Omega, \mathcal{B}, \mu)$ be a measure space, $I
			\subset \mathbb{R}$ an open interval and $f : I \times \Omega \rightarrow
			\mathbbm{R}$ a function such that:
			\begin{enumerateroman}
				\item $a \mapsto f (a, t)$ is differentiable for $\mu$-a.e. $t \in \Omega$
				
				\item $t \mapsto f (a, t)$ is $\mu$-integrable for all $a \in I$
				
				\item $\exists g \in L^1  (\Omega, \mathd \mu)$ such that $| \partial_a f
				(a, t) | \leq g (t)$ for all $a \in I$ and $\mu$-a.e. $t \in \Omega$
			\end{enumerateroman}
			Then, the function $a \mapsto \int_{\Omega} f (a, t) \mathd \mu (t)$ is
			differentiable and
			\[ \partial_a  \int_{\Omega} f (a, t) \mathd \mu (t) = \int_{\Omega}
			\partial_a f (a, t) \mathd \mu (t) \]
		\end{lemma}
		
		\begin{lemma}
			\label{lemmageratio}Let $I \subset \mathbb{R}$ be an open interval, $A
			\subset \mathbb{R}$ a non-empty Borel set, $\mu$ a $\sigma$-finite Borel
			measure on $A$ and $u, v : A \rightarrow [0, + \infty)$ measurable
			functions, not simultaneously $0$. Let $f : I \times A \rightarrow (0, +
			\infty)$ such that
			\begin{enumerateroman}
				\item $a \mapsto f (a, t)$ is differentiable for $\mu$-a.e. $t \in A$
				
				\item $t \mapsto u (t) f (a, t)$ and $t \mapsto v (t) f (a, t)$ are
				$\mu$-integrable for all $a \in I$.
				
				\item For each compact subset $K \subset I$, there exists a function $g_K
				: A \rightarrow [0, + \infty)$ such that $ug_K, vg_K$ are $\mu$-integrable
				and $| \partial_a f (a, t) | \leq g_K (t)$ for all $a \in K$ and
				$\mu$-a.e. $t \in A$.
			\end{enumerateroman}
			Let $F : I \rightarrow \mathbb{R}$ be defined by:
			\[ F (a) \assign \frac{\int_A f (a, t) u (t) \mathd \mu (t)}{\int_A f (a, t)
				v (t) \mathd \mu (t)} \]
			Then, the following hold:
			  \begin{enumerate}[I.]
				\item If for all $a \in I$ and for $\mu$-a.e. $t \in A$,  $\partial_a f (a, t) / f (a, t)$ and $u (t) / v
				(t)$ both increase or both decrease wrt $t$, then $F$ is increasing.
				\item If for all $a \in I$ and for $\mu$-a.e. $t \in A$, $\partial_a f (a, t) / f (a, t)$ increases (decreases) wrt
				$t$ and $u (t) / v (t)$ decreases (increases), then $F$ is
				decreasing.
			\end{enumerate}
		\end{lemma}
		
		\begin{proof}
			Let $U (a) = \int_A f (a, t) u (t) \mathd \mu (t)$, $V (a) = \int_A f (a, t)
			v (t) \mathd \mu (t)$. By the fact that $u (a) \partial_a f (a, t)$ and $v
			(t) \partial_a f (a, t)$ are dominated in compact subsets of $I$ by a
			$\mu$-integrable function of $t$, Lemma \ref{lemmaderiv} gives that both $U$
			and $V$ are differentiable, and the derivatives can be given by
			differentiating the integrands. Then, $F'$ also exists and hence we need to
			investigate the derivative
			
			\begin{align*}
			F' (a) = \frac{U' (a) V (a) - U (a) V' (a)}{V^2 (a)}
			\end{align*}			
			We find
			\begin{align*}
			& U' (a) V (a) - U (a) V' (a) =\\
			= & \int_A \int_A u (s) v (t)  (\partial_a f (a, s) f (a, t) - \partial_a
			f (a, t) f (a, s)) \mathd \mu (s) \mathd \mu (t)\\
			= & \int_A \int_{A \cap \{s < t\}} u (s) v (t)  (\partial_a f (a, s) f
			(a, t) - \partial_a f (a, t) f (a, s)) \mathd \mu (s) \mathd \mu (t)\\
			& \qquad + \int_A \int_{A \cap \{s > t\}} u (s) v (t)  (\partial_a f (a,
			s) f (a, t) - \partial_a f (a, t) f (a, s)) \mathd \mu (s) \mathd \mu
			(t)\\
			= & \int_A \int_{A \cap \{s < t\}} u (s) v (t)  (\partial_a f (a, s) f
			(a, t) - \partial_a f (a, t) f (a, s)) \mathd \mu (s) \mathd \mu (t)\\
			& \qquad + \int_A \int_{A \cap \{s < t\}} u (t) v (s)  (\partial_a f (a,
			t) f (a, s) - \partial_a f (a, s) f (a, t)) \mathd \mu (s) \mathd \mu
			(t)\\
			&= \int_A \int_{A \cap \{s < t\}} (u (s) v (t) - u (t) v (s)) 
			(\partial_a f (a, s) f (a, t) - \partial_a f (a, t) f (a, s)) \mathd \mu
			(s) \mathd \mu (t)
			\end{align*}			
			where in the pre-last equality we have made use of Fubini's theorem. The last
			integrand, as $s < t$, is non-negative (non-positive) if $\partial_a f / f$
			and $u / v$ have the same (opposite) monotonicity properties, which proves
			the lemma.
		\end{proof}
		
		\begin{remark}
			In the proceeding Lemma, the same conclusion holds if $u$, $v$ can assume
			the value zero at the same time, as then, without loss of generality, we can
			just integrate over the set $A' = A \setminus (\{u (t) = 0\} \cap \{v (t) =
			0\})$, which is again a Borel set, and we consider the condition $u / v$
			being increasing (or decreasing) in $A'$.
		\end{remark}
		
		\begin{remark}
			Lemma \ref{lemmageratio} is a general case of results concerning
			monotonicity properties of ratios of power series and polynomials. For
			instance, it gives \cite[Lemma 2.2]{henrik1}, if we set $\mu$ to be
			the counting measure on $\mathbb{N}$.
		\end{remark}
		
		\begin{lemma}
			\label{lemmaq}Let $I, J$ be two open intervals. Let $f : I \times J
			\rightarrow (0, \infty)$ such that:
			\begin{enumerateroman}
				\item $a \mapsto f (a, x)$ is differentiable for a.e. $x \in J$
				
				\item $x \mapsto f (a, x)$ is integrable for all $a \in I$
				
				\item For each compact subset $K \subset I$, there exists an integrable
				function $g_K : J \rightarrow [0, + \infty)$ such that $| \partial_a f (a,
				t) | \leq g_K (t)$ for all $a \in K$ and $\mu$-a.e. $t \in A$.
				
				\item The logarithmic derivative of $f$ wrt $a$ is increasing (decreasing)
				wrt $x$ for a.e. $x$, i.e.
				\[ \frac{\partial_a f (a, x)}{f (a, x)} \uparrow_x (\downarrow_x) \]
			\end{enumerateroman}
			Then, the $p$-quantile $q (a)$ of the probability distribution with density $f (a, x) / \int_J f (a, t) dt$ is increasing (decreasing) wrt $a$.
		\end{lemma}
		
		\begin{proof}
			We will deal with the case that the logarithmic derivative of $f$ is
			increasing, and the other case, that it is decreasing, is analogous. Let $x
			\in J = (c, d)$, where $- \infty \leq c < d \leq + \infty$. Then the
			cumulative distribution function is			
			\begin{align*}
			F (a ; x) = \frac{\int_c^x f (a, t) \mathd t}{\int_c^d f (a, t) \mathd t}
			= \frac{\int_c^d f (a, t) 1_{[c, x]} (t) \mathd t}{\int_c^d f (a, t)
				\mathd t}
			\end{align*}			
			We set $u (t) = 1_{[c, x]} (t)$ and $v (t) = 1$. As $u / v = u$ decreases
			and $\partial_a f / f$ increases wrt $t$, by Lemma \ref{lemmageratio} we get
			that $F$ decreases pointwise wrt $a$. This means			
			\begin{align*}
			\frac{\int^{q (a + h)}_c f (a, t) \mathd t}{\int_c^d f (a, t) \mathd t}
			\geq \frac{\int^{q (a + h)}_c f (a + h, t) \mathd t}{\int_c^d f (a + h, t)
				\mathd t} = p = \frac{\int^{q (a)}_c f (a, t) \mathd t}{\int_c^d f (a, t)
				\mathd t}
			\end{align*}			
			so that $q (a + h) \geq q (a)$ and hence that the $p$-quantile is
			increasing.
		\end{proof}
		
		\begin{remark}
			In Lemma \ref{lemmageratio}, if the logarithmic derivative $\partial_a f /
			f$ is strictly monotone (and $u \nequiv v$), it is easy to see from the
			proof that in the conclusion the ratio of the integrals should also be
			strictly monotone. Hence, also in Lemma \ref{lemmaq}, if the logarithmic
			derivative is strictly increasing (decreasing), then the $p$-quantile is
			also strictly increasing (decreasing).
		\end{remark}
		
		The following Lemma deals with the question of convergence of $p$-quantiles of
		a convergent sequence of probability distributions. We denote the extended
		real line $\mathbbm{R} \cup \{ \pm \infty \}$ by $\hat{\mathbbm{R}}$, with its
		usual topology.
		
		\begin{lemma}
			\label{lemmaasymp}Let $F_n : \hat{\mathbbm{R}} \rightarrow [0, 1]$ be a
			sequence of cumulative distribution functions on $\mathbbm{R}$, extended by
			$F_n (- \infty) \assign 0$ and $F_n (+ \infty) \assign 1$. Let $q_n$ be a
			$p$-quantile of $F_n$, i.e. $F_n (q_n) = p \in (0, 1)$, $\forall n \in
			\mathbbm{N}$. Assume the following conditions:
			\begin{enumerateroman}
				\item The sequence $(F_n (x))_{n \in \mathbbm{N}}$ converges pointwise to
				a limit $F_{\infty} (x) \assign \lim_{n \rightarrow \infty} F_n
				(x)$\label{cond1}
				
				\item The sequence of $p$-quantiles converges to a limit $q_{\infty}
				\assign \lim_{n \rightarrow \infty} q_n \in
				\hat{\mathbbm{R}}$\label{cond3}
			\end{enumerateroman}
			Then,
			\begin{equation}
			q_{\infty} \in [\sup \{ x \in \hat{\mathbbm{R}} |F_{\infty} (x) < p \},
			\inf \{ x \in \hat{\mathbbm{R}} |F_{\infty} (x) > p \}] \label{pquaincl}
			\end{equation}
			Thus, if $F_{\infty}$ is continuous, $q_{\infty}$ is a $p$-quantile of
			$F_{\infty}$.
		\end{lemma}
		
		\begin{proof}
			Let some $w \in \hat{\mathbbm{R}}$ such that $F_{\infty} (w) < p$. By
			condition (\ref{cond1}) we have that there is some $n_0 \in \mathbbm{N}$
			such that $\forall n > n_0 : F_n (w) < p = F_n (q_n)$. As each $F_n$ is
			non-decreasing, we have that $\forall n > n_0 : w < q_n$ and hence
			$q_{\infty} \geqslant w$. As this holds $\forall w \in \{ x \in
			\hat{\mathbbm{R}} |F_{\infty} (x) < p \}$, we get that $q_{\infty} \geqslant
			\sup \{ x \in \hat{\mathbbm{R}} |F_{\infty} (x) < p \}$. In a similar way we
			may prove that $q_{\infty} \leqslant \inf \{ x \in \hat{\mathbbm{R}}
			|F_{\infty} (x) > p \}$. In case $F_{\infty}$ is continuous, we have $[\sup
			\{ x \in \hat{\mathbbm{R}} |F_{\infty} (x) < p \}, \inf \{ x \in
			\hat{\mathbbm{R}} |F_{\infty} (x) > p \}] = \{ x \in \mathbbm{R}|F_{\infty}
			(x) = p \}$, hence then $q_{\infty}$ is a $p$-quantile of $F_{\infty}$.
		\end{proof}
		
		\begin{remark}
			\label{remarkasymp}As $\hat{\mathbbm{R}}$ is compact, $p$-quantiles always have limit points, and the above Lemma shows
			that convergence of distribution functions for which $p$-quantiles
			exist implies that all their limit points lie in the interval in
			(\ref{pquaincl}). This interval either consists of the closure of $F_{\infty}^{-1}(\{p\})$, or, if this set is empty, it degenerates to a
			point, which is a point of discontinuity of $F_{\infty}$.
		\end{remark}
		
		\begin{lemma}
			\label{real}Let $I,J \subset \mathbb{R}$ be open intervals, and $(F (a ; x))_{a
				\in I}$ be a family of cumulative probability distribution functions of $x$ on $J$, having positive densities $f (a ;
			t)$ with respect to Lebesgue measure. Moreover assume that the corresponding
			densities are real analytic in both variables. Denote the respective
			$p$-quantiles by $q (a)$. Then, $q$ is a real analytic function of $a$.
		\end{lemma}		
		\begin{proof}
			As the densities are positive functions, the $p$-quantile exists and is
			unique for each $a$. Hence, the function $q (a)$ is well defined implicitly
			as the solution $y = q (a)$ to the equation $F (a ; y) - p = 0$. Let some
			$y_0 \in J$ and $a_0 \in I$ such that $F (a_0 ; y_0) - p = 0$. As $F$
			is real analytic and $\partial_y F (a ; y) = f (a ; y) \neq 0$, by
			{\cite[Theorem 6.1.2]{implicit}} the equation $F (a ; y) - p = 0$ has a real
			analytic solution $y = y (a)$ in a neighbourhood of $a_0$ such that $F (a_0
			; y (a_0)) - p = 0$. By uniqueness of the $p$-quantile this solution must be
			exactly $q (a)$, and hence $q$ is real analytic.
		\end{proof}		
		\section{Monotonicity and limits}				
		\begin{proof*}{Proof of Proposition \ref{prop1}}
			Fix $b>0$. As the regularised incomplete beta function $I (x ; a, b)$ is real
			analytic in $x$ and $a$, Lemma \ref{real} gives real analyticity of $q$. Let
			$\beta (a ; x) \assign x^{a - 1}  (1 - x)^{b - 1}$. Its logarithmic
			derivative wrt $a$ is			
			\begin{align*}
			\frac{\partial_a \beta (a, b ; x)}{\beta (a, b ; x)} = \frac{x^{a - 1}  (1
				- x)^{b - 1} \log x}{x^{a - 1}  (1 - x)^{b - 1}} = \log x
			\end{align*}			
			which is an increasing function of $x$ and Lemma \ref{lemmaq} gives us that
			$q$ is also increasing. Its limits at $0$ and $\infty$ are classical
			results. They can also be obtained by considering limits of the incomplete
			beta function and using Lemma \ref{lemmaasymp}. Let, for instance, some
			limit point $\lim_{n \rightarrow \infty} q (a_n) = q_{\infty} \in [0, 1]$
			for a sequence $a_n \rightarrow \infty$. Then, the fact that $\lim_{a
				\rightarrow \infty} I (x ; a, b)$ vanishes for $x\in [0,1)$ and is a unit at $x=1$ gives $q_{\infty} = 1$, hence $\lim_{a \rightarrow
				\infty} q (a) = 1$. A similar argument shows $\lim_{a \rightarrow 0} q (a) =
			0$.
		\end{proof*}
		
		\begin{proof*}{Proof of Proposition \ref{prop2}}
			By Proposition \ref{prop1} already, $\phi$ can be seen to be a real analytic
			function. Regarding monotonicity, if $b = 1$ then $\phi (a) \equiv - \log
			p$. Assume $b > 1$. By using a change of variables on (\ref{betamaineq}) we
			get
			\begin{equation}
			\int_{\phi (a)}^{\infty} e^{- s}  (1 - e^{- s / a})^{b - 1} \mathd s = p
			\int_0^{\infty} e^{- s}  (1 - e^{- s / a})^{b - 1} \mathd s \label{eqphi}
			\end{equation}
			and hence the function $\phi$ is the $(1 - p)$-quantile of the distribution
			with density function		
			\begin{align}
			x \mapsto \frac{e^{- x}  (1 - e^{- x / a})^{b - 1}}{\int_0^{+ \infty} e^{-
					s}  (1 - e^{- s / a})^{b - 1} \mathd s}
			\end{align}			
			We set $f (a ; x) \assign e^{- x}  (1 - e^{- x / a})^{b - 1}$. The
			logarithmic derivative of $f$ wrt $a$ is			
			\begin{align*}
			\frac{\partial_a f (a ; x)}{f (a ; x)} = - \frac{(b - 1) xe^{- x / a}}{a^2
				(1 - e^{- x / a})}
			\end{align*}			
			The derivative of this wrt $x$ is			
			\begin{align*}
			\partial_x \left( \frac{\partial_a f (a ; x)}{f (a ; x)} \right) = \frac{b
				- 1}{a^3} \mathrm{e}^{- \frac{x}{a}} \left( a \mathrm{e}^{- \frac{x}{a}} -
			a + x \right)  \left( - 1 + \mathrm{e}^{- \frac{x}{a}} \right)^{- 2} \geq
			0
			\end{align*}			
			as the function $x \mapsto a \mathrm{e}^{- \frac{x}{a}} - a + x$ has
			positive derivative for $x > 0$ and vanishes at $0$. Thus, by Lemma
			\ref{lemmaq} we have that $\phi$ is increasing. The case $b < 1$ is similar.
				
			For the asymptotic results, we notice that for $a \rightarrow 0$, we have
			that
			\[ \lim_{a \rightarrow 0} \frac{e^{- x}  (1 - e^{- x / a})^{b -
					1}}{\int_0^{\infty} e^{- s}  (1 - e^{- s / a})^{b - 1} \mathd s} =
			\frac{e^{- x}}{\int_0^{\infty} e^{- s} \mathd s} = e^{- x} \]
			The corresponding distributions, whose $p$-quantiles are equal to $\phi
			(a)$, converge to the gamma distribution with parameter $1$, and hence by
			Lemma \ref{lemmaasymp} $\lim_{a \rightarrow 0} \phi (a) = - \log p$.
			Similarly, for $a \rightarrow \infty$
			\[ \lim_{a \rightarrow \infty} \frac{e^{- x}  (1 - e^{- x / a})^{b -
					1}}{\int_0^{\infty} e^{- s}  (1 - e^{- s / a})^{b - 1} \mathd s} =
			\lim_{a \rightarrow \infty} \frac{e^{- x}}{\int_0^{\infty} e^{- s} 
				\frac{(1 - e^{- s / a})^{b - 1}}{(1 - e^{- x / a})^{b - 1}} \mathd s} =
			\frac{e^{- x} x^{b - 1}}{\int_0^{\infty} e^{- s} s^{b - 1} \mathd s} \]
			hence the distribution converges to the gamma distribution with parameter
			$b$ and $\lim_{a \rightarrow \infty} \phi (a) = \gamma_b$, the $(1 -
			p)$-quantile of the gamma distribution with parameter $b$.
		\end{proof*}
		
		\section{Convexity of $\phi$ for $b < 1$}		
		We rewrite (\ref{eqphi}) as
		\begin{equation}
		\int_0^{\phi (a)} e^{- s} (1 - e^{- s / a})^{b - 1} \mathd s = (1 - p)
		\int_0^{\infty} e^{- s} (1 - e^{- s / a})^{b - 1} \mathd s \label{phimaineq}
		\end{equation}
		We denote $f (a ; s) = e^{- s} (1 - e^{- s / a})^{b - 1}$ and
		differentiating the above equation we have
		\begin{equation}
		\phi' (a) f (a ; \phi (a)) + \int_0^{\phi (a)} \partial_1 f (a ; t) \mathd t
		= (1 - p) \int_0^{\infty} \partial_1 f (a ; t) \mathd t
		\end{equation}
		Differentiating again,
		\begin{align}
		\phi'' (a) f (a ; \phi (a))= & (1 - p) \int_{\phi (a)}^{\infty} \partial_1^2 f (a ; t) \mathd t - p \int_0^{\phi (a)} \partial_1^2 f (a ; t) \mathd t \nonumber\\
		&	
		- (\phi' (a))^2 \partial_2 f (a ; \phi (a)) - 2 \phi' (a) \partial_1 f (a ;
		\phi (a))   \label{eqsecder}
		\end{align}
		where $\partial_j, j \in \mathbbm{N}$, denotes differentiation wrt to the
		$j$th variable.\\

		\begin{proof*}{Proof of Theorem \ref{conv1}}
			Let $b \in (0, 1)$. By Proposition \ref{prop2} $\phi' < 0$, and as
			\[ \partial_s f (a ; s) = - e^{- s} (1 - e^{- s / a})^{b - 1} + \frac{b -
				1}{a} e^{- s} (1 - e^{- s / a})^{b - 2} e^{- s / a} < 0 \]
			and
			\[ \partial_a f (a ; s) = - s \frac{b - 1}{a^2} e^{- s} (1 - e^{- s / a})^{b
				- 2} > 0 \]
			we see that $\phi' (a)^2 \partial_2 f (a ; \phi (a)) < 0$ and $\phi'
			(a) \partial_1 f (a ; \phi (a)) < 0$. In order to show that $\phi'' > 0$,
			using (\ref{eqsecder}) what is left is to show that
			\begin{equation}
			(1 - p) \int_{\phi (a)}^{\infty} \partial_1^2 f (a ; t) \mathd t - p
			\int_0^{\phi (a)} \partial_1^2 f (a ; t) \mathd t \geqslant 0
			\label{eqd2f}
			\end{equation}
			
			We shall rewrite the above integrals in another way. We have
			\begin{align*}
			\int_{\phi (a)}^{\infty} \partial_1^2& f (a ; t) \mathd t=\\ = &
			\frac{b - 1}{a^4} \int_{\phi (a)}^{\infty} 2 t  e^{-\frac{2t}{a}} e^{- t}
			(1 - e^{- \frac{t}{a}})^{b - 3} \left(\left(a - \frac{t}{2}\right) e^{\frac{t}{a}} + \frac{(b - 1) t }{2} - a\right)
			\mathd t \nonumber\\
			= & \frac{2 (b - 1)}{a} \int_{\frac{\phi(a)}{a}}^{\infty} s e^{- a s} e^{- 2
				s} (1 - e^{- s})^{b - 3} \left( \left( 1 - \frac{s}{2} \right) e^s +
			\frac{b - 1}{2} s - 1 \right) \mathd s \\
			= & \frac{2 (b - 1)}{a} \int_{\frac{\phi(a)}{a}}^{\infty} e^{- a t} \eta_{}
			(t) \mathd t \nonumber
			\end{align*}
			where
			\begin{equation}\label{etadef}
			\eta_{_{}} (x) \assign x e^{- 2 x} (1 - e^{- x})^{b - 3} \left(  e^x- 1 -
			\frac{x}{2}  e^x + \frac{b - 1}{2} x  \right)
			\end{equation}
			and similarly
			\begin{align*}
				\int_0^{\phi (a)} \partial_1^2 f (a ; t) \mathd t = \frac{2 (b -
					1)}{a} \int_0^{\frac{\phi(a)}{a}} e^{- a t} \eta_{} (t) \mathd t
			\end{align*}
			Hence we can rewrite
			\begin{align}
			 (1 - p) \int_{\phi (a)}^{\infty}& \partial_1^2 f (a ; t) \mathd t - p
			\int_0^{\phi (a)} \partial_1^2 f (a ; t) \mathd t = \nonumber\\
			 &\frac{2 (b - 1)}{a} \left( (1 - p) \int_{\phi (a) / a}^{\infty} e^{-
				a t} \eta_{} (t) \mathd t - p \int_0^{\phi (a) / a} e^{- a t} \eta_{} (t)
			\mathd t \right)  \label{eqh}
			\end{align}
			
			We now proceed to show (\ref{eqd2f}). We see in Lemma \ref{lemmaeta} that
			the function
			\begin{equation}
			w (x) \assign \left( 1 - \frac{x}{2} \right) e^x + \frac{b - 1}{2} x - 1
			\label{www}
			\end{equation}
			has a unique root $\rho$ on $(0, + \infty)$, and it is positive on $(0,
			\rho)$ and negative on $(\rho, \infty)$. Assume that $\phi (a) \geqslant
			\rho a.$ As $w$ and $\eta$ have the same sign, we have that $\int_{\phi (a)
				/ a}^{\infty} e^{- a t} \eta_{} (t) \mathd t < 0$. For the other integral,
			we have
			\begin{align*}
				\int_0^{\phi (a) / a} e^{- a t} \eta_{} (t) \mathd t = & \int_0^{\rho}
				e^{- a t} \eta_{} (t) \mathd t + \int_{\rho}^{\phi (a) / a} e^{- a t}
				\eta_{} (t) \mathd t\\
				 \geq &  e^{- a \rho} \left( \int_0^{\rho} \eta_{} (t) \mathd t +
				\int_{\rho}^{\phi (a) / a} \eta_{} (t) \mathd t \right)\\
				 \geqslant & e^{- a \rho} \left( \int_0^{\rho} \eta_{} (t) \mathd t +
				\int_{\rho}^{\infty} \eta_{} (t) \mathd t \right) = e^{- a \rho}
				\int_0^{\infty} \eta_{} (t) \mathd t = 0
			\end{align*}
			by Lemma \ref{inth0}. Hence
			\[ \frac{2 (b - 1)}{a} \left( (1 - p) \int_{\phi (a) / a}^{\infty} e^{- a t}
			\eta_{} (t) \mathd t - p \int_0^{\phi (a) / a} e^{- a t} \eta_{} (t)
			\mathd t \right) \geqslant 0 \]
			and by (\ref{eqh}), (\ref{eqd2f}) is proved for $\phi (a) \geqslant \rho a.$
			
			Now, assume that $\phi (a) < \rho a$. We define
			\begin{equation}
			h (a ; t) \assign \frac{\partial_1^2 f (a ; t)}{(b - 1) f (a ; t)} =
			\frac{2 t ((a - t / 2) e^{t / a} + (b - 1) t / 2 - a)}{a^4 (e^{t / a} -
				1)^2}
			\end{equation}
			We further denote
			\begin{equation}
			h_0 (s) \assign \frac{a^2}{2} h (a ; a s) = \frac{s ((1 - s / 2) e^s + (b
				- 1) s / 2 - 1)}{(e^s - 1)^2} = \frac{s w (s)}{(e^s - 1)^2} \label{h0}
			\end{equation}
			By Lemma \ref{lemmaexppol}, $h_0$ is decreasing on $(0, \rho)$, hence $h (a
			; s)$ is also decreasing wrt $s$ on $(0, \rho a)$. Hence, for $t \in (0,
			\phi (a)) \subset (0, \rho a)$ we have $h (a ; t) > h (a ; \phi (a))$. For
			$t \in (\phi (a), \rho a)$, we analogously have $h (a ; \phi (a)) > h (a ;
			t)$, and if $t \in (\rho a, \infty)$, then $h (a ; \phi (a)) > 0 > h (a ;
			t)$. Hence,
			\begin{align*}
				  (1 - p) &\int_{\phi (a)}^{\infty} \partial_1^2 f (a ; t) \mathd t - p
				\int_0^{\phi (a)} \partial_1^2 f (a ; t) \mathd t=\\
				= & (b - 1) \left( (1 - p) \int_{\phi (a)}^{\infty} h (a ; t) f (a ; t)
				\mathd t - p \int_0^{\phi (a)} h (a ; t) f (a ; t) \mathd t \right)\\
				 \geqslant & (b - 1) h (a ; \phi (a)) \left( (1 - p) \int_{\phi
					(a)}^{\infty} f (a ; t) \mathd t - p \int_0^{\phi (a)} f (a ; t) \mathd t
				\right) = 0
			\end{align*}
			by (\ref{phimaineq}). Thus (\ref{eqd2f}) is proved. As the RHS of
			(\ref{eqsecder}) is positive, then $\phi'' > 0$.
		\end{proof*}
		
		\begin{lemma}
			\label{lemmaeta}Fix $b > 0$. The function $w$  in (\ref{www}) has a
			unique root $\rho$ on $(0, \infty)$. We have that $w (x) > 0$ for $x < \rho$
			and $w (x) < 0$ for $x > \rho$.
		\end{lemma}
		
		\begin{proof}
			We have
			\[ w' (x) = \frac{1 - x}{2} e^x + \frac{b - 1}{2} \]
			and
			\[ w'' (x) = - \frac{x}{2} e^x < 0 \quad \tmop{for} x > 0 \]
			Hence $w'$ is strictly decreasing, and as $w' (0) = b / 2$ and $\lim_{x
				\rightarrow + \infty} w' (x) = - \infty$, it changes its sign exactly once
			and we get that $w$ is initially increasing and then decreasing, concave
			function. As $w (0) = 0$ and $\lim_{x \rightarrow + \infty} w (x) = -
			\infty$, we get that $w$ has a unique root $\rho \in (0, \infty)$, and $w
			(x) > 0$ for $x < \rho$ and $w (x) < 0$ for $x > \rho$.
		\end{proof}
		
		\begin{lemma}
			For $b > 0$, it holds that\label{inth0}
			\[ \int_0^{\infty} s e^{- 2 s} (1 - e^{- s})^{b - 3} \left( e^s -  1 -
			\frac{s}{2}  e^s + \frac{b - 1}{2} s  \right) \mathd s = 0 \]
		\end{lemma}
		
		\begin{proof}
			In the course of the proof we assume that $b \neq 1, 2$, which may be lifted
			in the end by taking limits. We split the integral into 3 parts. The first
			one is
			\begin{align*}
				I_1 = & \int_0^{\infty} s e^{- 2 s} (1 - e^{- s})^{b - 3} (e^s - 1)
				\mathd s\\
				= & \int_0^{\infty} s e^{- s} (1 - e^{- s})^{b - 2} \mathd s\\
				= & - \int_0^{\infty} \log (1 - e^{- t}) e^{- (b - 1) t} \mathd t\\
				= & - \int_0^{\infty} \log (1 - e^{- t}) \left( \frac{1 - e^{- (b - 1)
						t}}{b - 1} \right)' \mathd t\\
				= & \frac{1}{b - 1} \int_0^{\infty} \frac{e^{- t} - e^{- b t}}{1 - e^{-
						t}} \mathd t\\
				= & \frac{\Psi (b) + \gamma}{b - 1}
			\end{align*}
			where $\Psi \assign \Gamma' / \Gamma$ is the digamma function (see \cite[Chapter 1]{special_functions}). For the
			second part,
			\begin{align*}
				I_2 = & \frac{b - 1}{2} \int_0^{\infty} s^2 e^{- 2 s} (1 - e^{- s})^{b -
					3} \mathd s\\
				= & \frac{b - 1}{2 (b - 2)} \left( \int_0^{\infty} s^2 e^{- s} (1 -
				e^{- s})^{b - 2} \mathd s - 2 \int_0^{\infty} s e^{- s} (1 - e^{- s})^{b -
					2} \mathd s \right)\\
				= & \frac{b - 1}{2 (b - 2)} \int_0^1 \log^2 t (1 - t)^{b - 2} \mathd t
				- \frac{\Psi (b) + \gamma}{b - 2}\\
				= & \frac{b - 1}{2 (b - 2)} \partial_1^2 \Beta (1, b - 1) - \frac{\Psi
					(b) + \gamma}{b - 2}
			\end{align*}
			using that $\partial_1^n \Beta (a, b) = \int_0^1 t^{a - 1} (1 - t)^{b - 1}
			\log^n t \mathd t$ for $b > - n$, which is derived by differentiating the
			integral representation of the beta function for $b > 0$ and using the
			identity principle. Finally,
			\begin{align*}
				I_3 = &  - \frac{1}{2} \int_0^{\infty} s^2 e^{- s} (1 - e^{- s})^{b - 3}
				\mathd s\\
				= &  - \frac{1}{2} \int_0^1 (\log t)^2 (1 - t)^{b - 3} \mathd t\\
				= &  - \frac{1}{2} \partial_1^2 \Beta (1, b - 2)\\
				= &  - \frac{1}{2} \partial_a^2 \left( \Beta (a, b - 1) \frac{a + b -
					2}{b - 2} \right) \bigg| _{a=1}\\
				= &  - \frac{b - 1}{2 (b - 2)} \partial_1^2 \Beta (1, b - 1) -
				\frac{\partial_1 \Beta (1, b - 1)}{b - 2}\\
				= &  - \frac{b - 1}{2 (b - 2)} \partial_1^2 \Beta (1, b - 1) +
				\frac{\gamma + \Psi (b)}{(b - 2) (b - 1)}
			\end{align*}
			where we have used that $\partial_1 \Beta (1, b - 1) = \frac{\gamma + \Psi
				(b)}{b - 1}$. We see that $I_1 + I_2 + I_3 = 0$, and the Lemma is proved.
		\end{proof}
		
		\begin{lemma}
			\label{lemmaexppol}Fix $b > 0$. The function $h_0$ in (\ref{h0}) is
			decreasing between $0$ and its root $\rho \in (0, \infty)$.
		\end{lemma}
		
		\begin{proof}
			It is easy to see that $x / (e^x - 1)$ is decreasing. The rest is also
			decreasing as
			\[ \frac{\left( 1 - \frac{x}{2} \right) e^x + \frac{b - 1}{2} x - 1}{e^x -
				1} = \frac{b}{2} \frac{x}{e^x - 1} + 1 - \frac{1}{2} \frac{x (e^x +
				1)}{e^x - 1} \]
			and
			\[ \left( \frac{x (e^x + 1)}{e^x - 1} \right)' = \frac{e^{2 \hspace{0.17em}
					x} - 2 \hspace{0.17em} e^x x - 1}{(e^x - 1)^2} \geqslant 0 \]
			as $\left( e^{2 \hspace{0.17em} x} - 2 \hspace{0.17em} e^x x - 1 \right)' =
			2 e^x (e^x - x - 1) \geqslant 0$ and the numerator vanishes at $0$. Hence,
			on $(0, \rho)$, $h_0$ is the product of two decreasing, positive functions,
			hence decreasing.
		\end{proof}
		
		\section{Logarithmic concavity of $q$}
		
		In this section, we shall prove Theorem \ref{logconc}. In order to have a more clear notation, we shall often
		denote the functions of $a$ ($q, \phi$ and $\psi$), without their argument.
		Using {\cite[8.17.7]{DLMF}}, we can rewrite (\ref{betamaineq}), as
		\begin{equation}
		\frac{q^a}{a} ^{} {}_2 F_1 (a, 1 - b ; a + 1 ; q) = p \frac{\Gamma (a)
			\Gamma (b)}{\Gamma (a + b)}
		\end{equation}
		and expanding the hypergeometric sum,
		\begin{equation}
		q^a \sum_{n = 0}^{\infty} \frac{(1 - b)_n q^n}{(a + n) n!} = p \Gamma (b)
		\frac{\Gamma (a)}{\Gamma (a + b)} \label{hyper1}
		\end{equation}
		Of course, if $b \in \mathbbm{N}$, the sum above terminates at $b - 1$, as $(1
		- b)_b = 0$. Using that
		\[ \frac{(1 - b)_n}{n!} = \frac{(b - 1) (b - 1 - 1) \cdots (b - 1 - n -
			1)}{n!} (- 1)^n = \binom{b - 1}{n} (- 1)^n \]
		and denoting
		\begin{equation}
		\psi \assign - \log q
		\end{equation}
		we can rewrite (\ref{hyper1}) further as
		\[ e^{- a \psi} \sum_{n = 0}^{\infty} \binom{b - 1}{n} \frac{(- 1)^n e^{- n
				\psi}}{(a + n)} = p \Gamma (b) \frac{\Gamma (a)}{\Gamma (a + b)} \]
		that is
		\begin{equation}
		\sum_{n = 0}^{\infty} \frac{\Gamma (a + b)}{\Gamma (a) (a + n)} \binom{b -
			1}{n} (- 1)^n e^{- (n + a) \psi} = p \Gamma (b) \label{eqpsi}
		\end{equation}
		We shall show that $\psi$ is convex, which shall imply the
		logarithmic concavity. The following lemma will be the key to this proof.
		
		\begin{lemma}
			We have that
			\begin{equation}
			\psi' = \sum_{n = 0}^{\infty} \frac{1}{a + b + n} Y_{n + b} (\psi) -
			\sum_{n = 0}^{\infty} \frac{1}{a + n} Y_n (\psi) \label{psi1}
			\end{equation}
			where

			\begin{equation}
			Y_c (\psi) \assign \frac{\int_0^{\psi} e^{c t} (1 - e^{- t})^{b - 1}
				\mathd t}{e^{c \psi} (1 - e^{- \psi})^{b - 1}}
			\end{equation}
		\end{lemma}
		
		\begin{proof}
			Differentiating (\ref{eqpsi}) we get
			\begin{align*}
				0 = & \sum_{n = 0}^{\infty} \binom{b - 1}{n} (- 1)^n e^{- (n + a) \psi}
				\left[ \left( \frac{\Gamma (a + b)}{\Gamma (a) (a + n)} \right)' - (\psi +
				(n + a) \psi') \frac{\Gamma (a + b)}{\Gamma (a) (a + n)} \right]\\
				= & e^{- a \psi} \sum_{n = 0}^{\infty} \binom{b - 1}{n} (- 1)^n e^{- n
					\psi} \left[ \left( \frac{\Gamma (a + b)}{\Gamma (a) (a + n)} \right)' -
				\psi' \frac{\Gamma (a + b)}{\Gamma (a)} - \psi \frac{\Gamma (a +
					b)}{\Gamma (a) (a + n)} \right]
			\end{align*}
			Using the fact that $\sum_{n = 0}^{\infty} \binom{b - 1}{n} (- 1)^n e^{- n
				\psi} = (1 - e^{- \psi})^{b - 1}$, we get
			
			\begin{align*}
				 \psi' & (1 - e^{- \psi})^{b - 1} =\\
				= & \sum_{n = 0}^{\infty} \binom{b - 1}{n} (- 1)^n e^{- n \psi} \left(
				\left( \frac{\Gamma (a + b)}{\Gamma (a) (a + n)}
					\right)'\Bigg/\left(\frac{\Gamma (a + b)}{\Gamma (a)}\right) - \frac{\psi}{a + n}
				\right)\\
				= & \sum_{n = 0}^{\infty} \binom{b - 1}{n} (- 1)^n e^{- n \psi} \left(
				\frac{\Psi (a + b) - \Psi (a)}{a + n} - \frac{1}{(a + n)^2} -
				\frac{\psi}{a + n} \right)\\
				= & \sum_{n = 0}^{\infty} \binom{b - 1}{n} (- 1)^n e^{- n \psi} \left(
				\sum_{k = 0}^{\infty} \left( \frac{1}{k + a} - \frac{1}{k + a + b} \right)
				\frac{1}{a + n} - \frac{1}{(a + n)^2} - \frac{\psi}{a + n} \right)\\
				= & \sum_{n = 0}^{\infty} \binom{b - 1}{n} (- 1)^n e^{- n \psi}
				\times\\
				&  \left( \sum_{k \neq n} \left( \frac{1}{(k + a) (a + n)} -
				\frac{1}{(k + a + b) (a + n)} \right) - \frac{1}{(n + a + b) (a + n)} -
				\frac{\psi}{a + n} \right)\\
				= & \sum_{n = 0}^{\infty} \binom{b - 1}{n} (- 1)^n e^{- n \psi}\times\\
				&\sum_{k \neq n} \left( \left( \frac{1}{a + n} - \frac{1}{a +
					k} \right) \frac{1}{k - n} - \left( \frac{1}{a + n} - \frac{1}{a + k + b}
				\right) \frac{1}{k + b - n} \right)\\
				&  -\sum_{n = 0}^{\infty} \left( \binom{b - 1}{n} (- 1)^n e^{- n \psi} \left( \frac{1}{a + n} - \frac{1}{a
					+ n + b} \right) \frac{1}{b} - \frac{\psi}{a + n}\right) \\
				= & \sum_{n = 0}^{\infty} \frac{1}{a + n} \binom{b - 1}{n} (- 1)^n e^{-
					n \psi} \left( \sum_{k \neq n} \left( \frac{1}{k - n} - \frac{1}{k + b -
					n} \right) \right)
				\\ & - \sum_{n = 0}^{\infty} \frac{1}{a + n} \binom{b -
					1}{n} (- 1)^n e^{- n \psi} \psi +
				\sum_{n = 0}^{\infty} \binom{b - 1}{n} (- 1)^n e^{- n \psi} \times
				\\ & \left(
				\sum_{k \neq n} \left( \frac{1}{a + k + b} \frac{1}{k + b - n} -
				\frac{1}{a + k} \frac{1}{k - n} \right) + \left( \frac{1}{a + n + b} -
				\frac{1}{a + n} \right) \frac{1}{b} \right)\\
				= & \sum_{n = 0}^{\infty} \frac{1}{a + n} \binom{b - 1}{n} (- 1)^n e^{-
					n \psi} \left( \sum_{k \neq n} \left( \frac{1}{k - n} - \frac{1}{k + b -
					n} \right) \right)\\
				& + \sum_{n = 0}^{\infty} \left( \frac{1}{a + n + b} - \frac{1}{a + n}
				\right) \binom{b - 1}{n} (- 1)^n e^{- n \psi} \frac{1}{b} - \sum_{n =
					0}^{\infty} \frac{1}{a + n} \binom{b - 1}{n} (- 1)^n e^{- n \psi} \psi\\
				& + \sum_{n = 0}^{\infty} \sum_{k \neq n} \binom{b - 1}{n} (- 1)^n e^{-
					n \psi} \left( \frac{1}{a + k + b} \frac{1}{k + b - n} - \frac{1}{a + k}
				\frac{1}{k - n} \right)\\
				\end{align*}
				\begin{align*}
				= & \sum_{n = 0}^{\infty} \frac{1}{a + n} \binom{b - 1}{n} (- 1)^n e^{-
					n \psi} \left( \sum_{k \neq n} \left( \frac{1}{k - n} - \frac{1}{k + b -
					n} \right) \right)\\
				& +  \sum_{n = 0}^{\infty} \left( \frac{1}{a + n + b} - \frac{1}{a + n}
				\right) \binom{b - 1}{n} (- 1)^n e^{- n \psi} \frac{1}{b} - \sum_{n =
					0}^{\infty} \frac{1}{a + n} \binom{b - 1}{n} (- 1)^n e^{- n \psi} \psi \\
				& +  \sum_{n = 0}^{\infty} \sum_{k \neq n} \binom{b - 1}{k} (- 1)^k e^{-
					k \psi} \left( \frac{1}{a + n + b} \frac{1}{n + b - k} - \frac{1}{a + n}
				\frac{1}{n - k} \right)\\
				= & \sum_{n = 0}^{\infty} \frac{1}{a + n} \binom{b - 1}{n} (- 1)^n e^{-
					n \psi} \left( \sum_{k \neq n} \left( \frac{1}{k - n} - \frac{1}{k + b -
					n} \right) - \frac{1}{b} \right) \\
				& +  \sum_{n = 0}^{\infty} \frac{1}{a + n} \left( \sum_{k \neq n}
				\binom{b - 1}{k} \frac{(- 1)^k e^{- k \psi}}{k - n} - \binom{b - 1}{n} (-
				1)^n e^{- n \psi} \psi \right) \\
				&  + \sum_{n = 0}^{\infty} \frac{1}{a + n + b} \sum_{k \neq n} \binom{b -
					1}{k} \frac{(- 1)^k e^{- k \psi}}{n + b - k} + \sum_{n = 0}^{\infty}
				\frac{1}{a + n + b} \binom{b - 1}{n} (- 1)^n e^{- n \psi} \frac{1}{b}\\
				= & \sum_{n = 0}^{\infty} \frac{1}{a + n} \binom{b - 1}{n} (- 1)^n e^{-
					n \psi} \left( \sum_{k \neq n} \left( \frac{1}{k - n} - \frac{1}{k + b -
					n} \right) - \frac{1}{b} \right) \\
				& +  \sum_{n = 0}^{\infty} \frac{1}{a + n} \left( \sum_{k \neq n}
				\binom{b - 1}{k} \frac{(- 1)^k e^{- k \psi}}{k - n} - \binom{b - 1}{n} (-
				1)^n e^{- n \psi} \psi \right) \\
				& +  \sum_{n = 0}^{\infty} \frac{1}{a + n + b} \sum_{k = 0}^{\infty}
				\binom{b - 1}{k} \frac{(- 1)^k e^{- k \psi}}{n + b - k}
			\end{align*}
			Thus we have
			\begin{align*}
				\psi' = & \sum_{n = 0}^{\infty} \frac{1}{a + n} X_n (\psi) + \sum_{n =
					0}^{\infty} \frac{1}{a + b + n} Z_n (\psi)
			\end{align*}
			where
			\begin{align*}
				X_n &(\psi) \assign  \Bigg[ \binom{b - 1}{n} (- 1)^n e^{- n \psi} \left(
				\sum_{k \neq n} \left( \frac{1}{k - n} - \frac{1}{k + b - n} \right) -
				\frac{1}{b} \right)\\ & + \sum_{k \neq n} \binom{b - 1}{k} (- 1)^k e^{- k
					\psi} \frac{1}{k - n} - \binom{b - 1}{n} (- 1)^n e^{- n \psi} \psi \Bigg]
				/ (1 - e^{- \psi})^{b - 1}
			\end{align*}
			and
			\begin{align*}
				Z_n (\psi)  \assign  \left( \sum^{\infty}_{k = 0} \binom{b - 1}{k} (-
				1)^k e^{- k \psi} \frac{1}{n + b - k} \right) / (1 - e^{- \psi})^{b - 1}
			\end{align*}
			By Lemma \ref{leema1} and
			\[ \partial_{\psi} \left( \sum^{\infty}_{k = 0} \binom{b - 1}{k} (- 1)^k
			e^{(n + b - k) \psi} \frac{1}{n + b - k} \right) = e^{(n + b) \psi} (1 -
			e^{- \psi})^{b - 1} \]
			we have that
			\[ \sum^{\infty}_{k = 0} \binom{b - 1}{k} (- 1)^k e^{(n + b - k) \psi}
			\frac{1}{n + b - k} = \int_0^{\psi} e^{(n + b) t} (1 - e^{- t})^{b - 1}
			\mathd t \]
			and hence we get
			\begin{align*}
				Z_n (\psi) = & e^{- (n + b) \psi} \left( \sum^{\infty}_{k = 0} \binom{b
					- 1}{k} (- 1)^k e^{(n + b - k) \psi} \frac{1}{n + b - k} \right) / (1 -
				e^{- \psi})^{b - 1}\\
				= & e^{- (n + b) \psi} \int_0^{\psi} e^{(n + b) t} (1 - e^{- t})^{b -
					1} \mathd t / (1 - e^{- \psi})^{b - 1}\\
				= & \frac{\int_0^{\psi} e^{(n + b) t} (1 - e^{- t})^{b - 1} \mathd
					t}{e^{(n + b) \psi} (1 - e^{- \psi})^{b - 1}} = Y_{n + b} (\psi)
			\end{align*}
			Similarly, Lemma \ref{leema1} and
			\[ \partial_{\psi} \left( \sum_{k \neq n} \binom{b - 1}{k} (- 1)^k e^{- (k -
				n) \psi} \frac{1}{k - n} - \binom{b - 1}{n} (- 1)^n \psi \right) = - e^{n
				\psi} (1 - e^{- \psi})^{b - 1} \]
			give
			\begin{align*}
				X_n (\psi) = & - \frac{\int_0^{\psi} e^{n t} (1 - e^{- t})^{b - 1}
					\mathd t}{e^{n \psi} (1 - e^{- \psi})^{b - 1}} = - Y_n (\psi)
			\end{align*}
			hence (\ref{psi1}) is proved.
		\end{proof}
		
		\begin{lemma}
			For $n \in \mathbbm{N}$ and $b > 0$, we have\label{leema1}
			\begin{equation}
			\sum^{\infty}_{k = 0} \binom{b - 1}{k} \frac{(- 1)^k}{n + b - k} = 0
			\label{sum1}
			\end{equation}
			and
			\begin{equation}
			\sum_{k \neq n} \binom{b - 1}{k} (- 1)^k \frac{1}{k - n} = - \binom{b -
				1}{n} (- 1)^n \left( \sum_{k \neq n} \left( \frac{1}{k - n} - \frac{1}{k +
				b - n} \right) - \frac{1}{b} \right) \label{sum2}
			\end{equation}
		\end{lemma}
		
		\begin{proof}
			For $z \in \mathbbm{C} \setminus \{ 0, - 1, - 2, \ldots \}$ we have,
			applying {\cite[Theorem 2.2.2]{special_functions}},
			\begin{align*}
			 \sum^{\infty}_{k = 0} \binom{b - 1}{k} \frac{(- 1)^k}{k + z} = &
			\frac{1}{z} \sum^{\infty}_{k = 0} \frac{(1 - b)_k (z)_k}{k! (z + 1)_k} =
			\frac{{}_2 F_1 (1 - b, z ; z + 1 ; 1)}{z} \\ = & \frac{\Gamma (z + 1) \Gamma
				(b)}{z \Gamma (z + b)} = \frac{\Gamma (z) \Gamma (b)}{\Gamma (z + b)} \end{align*}
			Hence, we get
			\[ \sum^{\infty}_{k = 0} \binom{b - 1}{k} \frac{(- 1)^k}{n + b - k} =
			\lim_{z \rightarrow n} \sum^{\infty}_{k = 0} \binom{b - 1}{k} \frac{(-
				1)^k}{z + b - k} = - \lim_{z \rightarrow n} \frac{\Gamma (b) \Gamma (- z
				- b)}{\Gamma (- z)} = 0 \]
			proving (\ref{sum1}). For (\ref{sum2}), assume $z \in \mathbbm{C} \setminus
			\mathbbm{N}$ and let
			\begin{align*}
				&   \binom{b - 1}{n} (- 1)^n \left( \sum_{k \neq n} \left( \frac{1}{k -
					z} - \frac{1}{k + b - z} \right) - \frac{1}{n + b - z} \right) + \sum_{k
					\neq n} \binom{b - 1}{k} \frac{(- 1)^k}{k - z}\\
				= & \binom{b - 1}{n} (- 1)^n \left( \sum^{\infty}_{k = 0} \left(
				\frac{1}{k - z} - \frac{1}{k + b - z} \right) - \frac{1}{n - z} \right) +
				\sum_{k = 0}^{\infty} \binom{b - 1}{k} \frac{(- 1)^k}{k - z} - \binom{b -
					1}{n} \frac{(- 1)^n}{n - z}\\
				= & \binom{b - 1}{n} (- 1)^n \left( \Psi (b - z) - \Psi (- z) -
				\frac{1}{n - z} \right) + \frac{\Gamma (b) \Gamma (- z)}{\Gamma (b - z)} -
				\binom{b - 1}{n} \frac{(- 1)^n}{n - z}\\
				= & \binom{b - 1}{n} (- 1)^n \left( \Psi (b - z) - \Psi (1 + z) - \pi
				\frac{\cos (\pi z)}{\sin (\pi z)} - \frac{1}{n - z} \right) + \frac{\Gamma
					(b) \Gamma (- z)}{\Gamma (b - z)} - \binom{b - 1}{n} \frac{(- 1)^n}{n -
					z}
			\end{align*}
			where we have used the reflection formula for the digamma function. We have,
			using that $\tmop{Res} \left( \frac{\pi \cos \pi z}{\sin \pi z}, n \right) =
			1$, that
			\[ \lim_{z \rightarrow n} \left( \Psi (b - z) - \Psi (1 + z) - \pi
			\frac{\cos (\pi z)}{\sin (\pi z)} - \frac{1}{n - z} \right) = \Psi (b -
			n) - \Psi (1 + n) \]
			Furthermore, using de L'H{\^o}pital's rule, we get
			\begin{align*}
				  \lim_{z \rightarrow n} &\left( \frac{\Gamma (b) \Gamma (- z)}{\Gamma
					(b - z)} - \binom{b - 1}{n} \frac{(- 1)^n}{n - z} \right) =\\
				= & \lim_{z \rightarrow n} \left( - \frac{\Gamma (b)}{\Gamma (b - z)
					\Gamma (1 + z)} \frac{\pi}{\sin \pi z} - \binom{b - 1}{n} \frac{(- 1)^n}{n
					- z} \right)\\
				= & \lim_{z \rightarrow n} \frac{\frac{\Gamma (b)}{\Gamma (b - z)
						\Gamma (1 + z)} (n - z) - \binom{b - 1}{n} (- 1)^n \frac{\sin \pi
						z}{\pi}}{\frac{\sin \pi z}{\pi} (n - z)}\\
				= & \lim_{z \rightarrow n} \frac{\left( \frac{\Gamma (b)}{\Gamma (b -
						z) \Gamma (1 + z)} \right)' (n - z) + \frac{\Gamma (b)}{\Gamma (b - z)
						\Gamma (1 + z)} + \binom{b - 1}{n} (- 1)^n \cos \pi z}{\frac{\sin \pi
						z}{\pi} - (n - z) \cos \pi z}\\
				= & \lim_{z \rightarrow n} \frac{\left( \frac{\Gamma (b)}{\Gamma (b -
						z) \Gamma (1 + z)} \right)'' (n - z) + \left( \frac{\Gamma (b)}{\Gamma (b
						- z) \Gamma (1 + z)} \right)' + \pi \binom{b - 1}{n} (- 1)^n \sin \pi
					z}{\cos \pi z + (n - z) \pi \sin \pi z + \cos \pi z}\\
				= & \frac{\frac{\Gamma (b)}{\Gamma (b - n) \Gamma (1 + n)} (\Psi (1 +
					n) - \Psi (b - n))}{2 \cos \pi n}\\
				= & (- 1)^n \binom{b - 1}{n} (\Psi (1 + n) - \Psi (b - n))
			\end{align*}
			hence getting (\ref{sum2}).
		\end{proof}
		
		\begin{lemma}
			Let $b > 1$ and $c > 0$. Then,\label{lemmaH} $Y_c$ is increasing on $(0,
			\infty)$. Moreover, $Y_c (x)$, $Y^{'}_c (x)$ are decreasing wrt $c$ for
			fixed $x$.
		\end{lemma}
		
		\begin{proof}
			We rewrite
			\begin{align*}
				\frac{\int_0^x e^{c t} (1 - e^{- t})^{b - 1} \mathd t}{e^{c x} (1 - e^{-
						x})^{b - 1}} = & \int_0^x e^{c (t - x)} \left( \frac{1 - e^{- t}}{1 -
					e^{- x}} \right)^{b - 1} \mathd t\\
				= & \int_0^x e^{c (t - x)} \left( \frac{e^x - e^{x - t}}{e^x - 1}
				\right)^{b - 1} \mathd t\\
				= & \int_0^x e^{- c v} \left( \frac{e^x - e^v}{e^x - 1} \right)^{b - 1}
				\mathd v
			\end{align*}
			Differentiating, we get
			\begin{align*}
				\Bigg( \int_0^x e^{- c v} \left( \frac{e^x - e^v}{e^x - 1} \right)^{b - 1}
				\mathd t \Bigg)'=& \int_0^x e^{- c v} \partial_x \left( \frac{e^x -
					e^v}{e^x - 1} \right)^{b - 1} \mathd v \\
				= & \int_0^x e^{- c v + x} (b - 1) \left( \frac{e^x - e^v}{e^x - 1}
				\right)^{b - 2}  \frac{e^v - 1}{(e^x - 1)^2}  \mathd v
			\end{align*}
			and this completes the proof.
		\end{proof}
		
		\begin{proof*}{Proof of Theorem \ref{logconc}}
			We shall show the convexity of $\psi = - \log q$, which is equivalent to logarithmic
			concavity of $q$. The case $b < 1$ is given by Theorem \ref{conv1}, as $a
			\psi'' = \phi'' - 2 \psi' > 0$. For $b = 1$, we have $\psi = \frac{\log (1 /
				p)}{a}$ hence $\psi'' = 0$. For $b > 1$, differentiating (\ref{psi1}) we get
			\begin{align*}
				\psi'' = & \sum_{n = 0}^{\infty} \frac{1}{(a + n)^2} Y_n (\psi) -
				\sum_{n = 0}^{\infty} \frac{1}{(a + b + n)^2} Y_{n + b} (\psi) +\\
				&  \left( \sum_{n = 0}^{\infty} \frac{1}{a + b + n} Y'_{n + b} (\psi) -
				\sum_{n = 0}^{\infty} \frac{1}{a + n} Y'_n (\psi) \right) \psi' > 0
			\end{align*}
			using that $\psi' < 0$ and Lemma \ref{lemmaH}.
		\end{proof*}
		
		\begin{remark}
			We notice that (\ref{psi1}) also gives
			\begin{equation}
			q' = \sum_{n = 0}^{\infty} \frac{1}{a + b + n} \frac{\int_q^1 t^{- n
					- b - 1} (1 - t)^{b - 1} \mathd t}{q^{- n - b - 1} (1 - q)^{b - 1}} -
			\sum_{n = 0}^{\infty} \frac{1}{a + n} \frac{\int_q^1 t^{- n - 1} (1 -
				t)^{b - 1} \mathd t}{q^{- n - 1} (1 - q)^{b - 1}}
			\end{equation}
		\end{remark}
		
		\appendix\section*{Appendix}
		
		Finally, we want to see how the $p$-quantile depends on the second parameter
		of the beta distribution. For clarity, from now on we denote the $p$-quantile
		of the beta distribution with parameters $a$ and $b$ by $q_p (a, b)$. We shall
		consider $a$ constant, and try to relate $q$ as a function of $b$ with the
		previous results.	
			
		A simple change of variables $s = 1 - t$ gives the functional relation
		
		\begin{align}
		I (x ; a, b) = 1 - I (1 - x ; b, a)
		\end{align}		
		which implies
		\[ p = I (q_p (a, b) ; a, b) = 1 - I (q_p (a, b) ; b, a) \Rightarrow I (q_p
		(a, b) ; b, a) = 1 - p = I (q_{1 - p} (a, b) ; b, a) \]
		and using the uniqueness of the $p$-quantile we get
		
		\begin{align}
		q_p (a, b) = 1 - q_{1 - p}  (b, a)
		\end{align}		
		Hence, by Proposition \ref{prop1}, we get that $q_p$ is decreasing in $b$ and		
		\begin{align*}
		\lim_{b \rightarrow 0} q_p (a, b) = 1
		\end{align*}		
		\begin{align*}
		\lim_{b \rightarrow \infty} q_p (a, b) = 0
		\end{align*}		
		Moreover, we have		
		\begin{align}
		(1 - q_p (a, b))^b = q_{1 - p}  (b, a)^b = e^{- \varphi_{1 - p} (b)}
		\end{align}		
		where $\varphi_{1 - p} (b) = - b \log q_{1 - p} (b, a)$, hence the behaviour of
		$q_p (a, b)$ as a function of $b$ can again be studied similarly through the
		function $\varphi_p$. We also easily see that $b \mapsto 1 - q_p (a, b)$ is
		log-concave. We remark that numerical evidence shows that $b \mapsto q_p (a,
		b)$ itself is not (log-)concave/convex. However, the function $b \mapsto
		\varphi_p (b)$ seems to be convex.		
		\section*{}
		\begin{acknowledgements}I would like to thank H.L.Pedersen for careful reading of the original manuscript and useful suggestions.
			\end{acknowledgements}
		
\bibliographystyle{plain}
\bibliography{betabib}

\end{document}